\documentclass[12pt,reqno]{amsart}
\usepackage{hyperref}
\usepackage{color, bm, amscd, tikz-cd}
\newcommand{\cB}{{\mathcal B}}

\newcommand{\cD}{{\mathcal D}}

\newcommand{\cF}{{\mathcal F}}

\newcommand{\cH}{{\mathcal H}}

\newcommand{\cK}{{\mathcal K}}

\newcommand{\cU}{{\mathcal U}}

\theoremstyle{plain}

\newtheorem{thm}{Theorem}[section]

\newtheorem{corollary}[thm]{Corollary}
\newtheorem{lemma}[thm]{Lemma}

\newtheorem{proposition}[thm]{Proposition}

\newtheorem{remark}[thm]{Remark}

\newtheorem{example}[thm]{Example}
\numberwithin{equation}{section}

\def\textmatrix#1&#2\\#3&#4\\{\bigl({#1 \atop #3}\ {#2 \atop #4}\bigr)}
\def\dispmatrix#1&#2\\#3&#4\\{\left({#1 \atop #3}\ {#2 \atop #4}\right)}
\numberwithin{equation}{section}

\def\textmatrix#1&#2\\#3&#4\\{\bigl({#1 \atop #3}\ {#2 \atop #4}\bigr)}
\def\dispmatrix#1&#2\\#3&#4\\{\left({#1 \atop #3}\ {#2 \atop #4}\right)}

\begin{document}
\title[Rational dilation of tetrablock contractions]{Rational dilation of tetrablock contractions revisited}
\author[J. A. Ball]{Joseph A. Ball}
\address{Department of Mathematics, Virginia Tech, Blacksburg, VA 24061-0123, USA\\ joball@math.vt.edu}
\author[H. Sau]{Haripada Sau}
\address{Department of Mathematics, Virginia Tech, Blacksburg, VA 24061-0123, USA\\ sau@vt.edu, haripadasau215@gmail.com}
\thanks{This research of the second-named is supported by SERB Indo-U.S. Postdoctoral Research Fellowship 2017/139.}

\begin{abstract}
A classical result of Sz.-Nagy asserts that a Hilbert-space contraction operator $T$ can be lifted to an isometry
$V$. A more general multivariable setting of recent interest for these ideas is the case where
(i) the unit disk is replaced by a certain domain contained in ${\mathbb C}^3$ (called the {\em tetrablock}),
(ii) the contraction operator $T$ is replaced by a commutative triple $(T_1, T_2, T)$ of
Hilbert-space operators having ${\mathbb E}$ as a
spectral set (a tetrablock contraction) .  The rational dilation question for this setting is whether a
tetrablock contraction $(T_1, T_2, T)$ can be lifted to a tetrablock isometry $(V_1, V_2, V)$ (a
commutative operator tuple which extends to a tetrablock-unitary tuple $(U_1, U_2, U)$---a
commutative tuple of normal operators with joint spectrum
contained in the distinguished boundary of the tetrablock).
We discuss necessary conditions for a tetrablock contraction to have a tetrablock-isometric lift.
We present an example of a tetrablock contraction which does have a tetrablock-isometric lift but violates
a condition previously thought to be necessary for the existence of such a lift. Thus the question of
whether a tetrablock contraction always has a tetrablock-isometric lift appears to be unresolved
 at this time.
 \end{abstract}

\maketitle

\section{Introduction} \label{S:intro}
A seminal development for future progress in nonselfadjoint operator theory is the Sz.-Nagy dilation theorem:
{\em given a contraction operator $T$ on a Hilbert space $\cH$, there is a unitary operator $\widetilde\cU$ on a larger
 Hilbert space $\widetilde \cK$ so that $T^n = P_\cH \widetilde \cU^n|_\cH$ for all $n \in {\mathbb Z}_+$.}
 This result provides a geometric explanation for the von Neumann inequality:  {\em for any Hilbert-space
 contraction operator $T$ and any polynomial $p$, $\| p(T) \| \le \sup_{z \in {\mathbb D}} | p(z) |$}.
 Here $\| p(T) \|$ is the {\em operator norm} $\| p(T)_{\cB(\cH)}$ of $p(T)$ as an element of the Banach algebra
$\cB(\cH)$ of bounded linear operators on the Hilbert space $\cH$.
 In modern language, we say that $\widetilde \cU$ in the Sz.-Nagy dilation theorem is a {\em dilation} of $T$,
 and that the content of the von Neumann inequality is that the unit disk ${\mathbb D}$ is a {\em spectral set}
 for any contraction operator $T$.
 Arveson \cite{ArvesonII} shortly afterwards formulated a multivariable version of this connection between
 the von Neumann inequality and dilation theory, called the {\em rational dilation problem}.
 Fix  a domain  $\Omega$ with compact closure contained in $d$-dimensional Euclidean space ${\mathbb C}^d$.
 Suppose that we are given a commutative  tuple ${\mathbf T} = (T_1, \dots, T_d)$  of
 Hilbert-space operators with Taylor spectrum contained in $\Omega$;  let us note here that for the case where
 $(T_1, \dots, T_d)$ consists of commuting matrices,  the Taylor spectrum amounts to the subset of ${\mathbb C}^d$ consisting
 of the joint eigenvalues of $(T_1, \dots, T_d)$.  If $r$ is any
 function holomorphic on $\Omega$, any reasonable functional calculus (see \cite{Curto} for a survey)
 can be used to define $r({\mathbf T})$; in case $\Omega$ is polynomially convex, it suffices to consider the case where
 $r$ is a polynomial.  We denote by $\operatorname{Rat}(\Omega)$ the algebra of all rational functions holomorphic on
 $\Omega$. We say that {\em $\Omega$ is a spectral set for ${\mathbf T}$} if for all $r \in \operatorname{Rat}(\Omega)$
 it is the case that
 $$
     \| r(\mathbf T) \|_{\cB(\cH)}\leq \sup \{ |r(z)| \colon z \in \Omega \}.
 $$
 Let us say that the operator tuple ${\mathbf U}  = (U_1, \dots, U_d)$ is {\em$\Omega$-unitary}
 if ${\mathbf U}$ is a commutative tuple of normal operators with joint spectrum contained in the distinguished
 boundary $\partial_e \Omega$ of $\Omega$. We say that ${\mathbf T}$ has an {\em $\Omega$-unitary dilation} if there is a $\Omega$-unitary
 tuple  ${\mathbf U} = (U_1, \dots, U_d)$ on a larger Hilbert space $\cK \supset \cH$ such that
 $r({\mathbf T} )= P_\cH r({\mathbf N})|_\cH$ for all $r \in \operatorname{Rat} \Omega$.
 If ${\mathbf T}$ has a $\Omega$-unitary dilation ${\mathbf U}$, it follows that
 \begin{align*}
 \| r({\mathbf T}) \| & = \| P_\cH r({\mathbf U}) |_\cH \|   \le \| r({\mathbf U}) \| \\
 & = \sup_{z \in \partial_e \Omega } \{ |r(z) | \} \text{ (by the functional calculus for commuting normal operators)} \\
 & = \sup_{z \in \Omega} \{| r(z) | \} \text{ (by the definition of the distinguished boundary) }
 \end{align*}
 and it follows that ${\mathbf T}$ has $\Omega$ as a spectral set. The {\em rational dilation problem} asks:
 for a given domain $\Omega$, when is it the case that the converse direction holds, i.e., that
 $\Omega$ being a spectral set for ${\mathbf T}$ implies that ${\mathbf T}$ has an $\Omega$-unitary
 dilation?

 The problem can be reformulated in terms of $\Omega$-isometric lifts rather than $\Omega$-unitary dilations
 as follows.  For ${\mathbf V} = (V_1, \dots, V_d)$ a commutative operator tuple on a Hilbert space $\cK$,
 we say that ${\mathbf V}$ is an {\em $\Omega$-isometry} if there is an $\Omega$-unitary operator tuple
 ${\mathbf U} = (U_1, \dots, U_d)$ on a larger space $\widetilde \cK \supset \cK$ which extends
 ${\mathbf V}$, i.e., such that $V_j = U_j|_\cK$ for $j=1, \dots, d$.  Given a commutative operator tuple
 ${\mathbf T} = (T_1, \dots, T_d)$ on $\cH$, we say that ${\mathbf T}$ has an {\em $\Omega$-isometric lift}
 if there is an $\Omega$-isometry ${\mathbf V} = (V_1, \dots, V_d)$ on $\cK \supset \cH$ such that
 ${\mathbf V}^*$ extends ${\mathbf T}^*$:  $T_j^* = V_j^*|_\cH$ for each $j=1, \dots, d$.  If ${\mathbf U}$
on $\widetilde \cK$  is an $\Omega$-unitary dilation of ${\mathbf T}$ on $\cH$, we may define
 $$
 \cK = \overline{\operatorname{span}} \{ r({\mathbf U}) h \colon h \in \cH, \, r \in \operatorname{Rat}(\Omega) \}
 $$
 and then set ${\mathbf V} = {\mathbf U}|_\cK$.  For any $r \in \operatorname{Rat}(\Omega)$ and $h \in \cH$,
 we have
 $$
  P_{\cK \ominus \cH} r({\mathbf U}) h = r({\mathbf U}) h - r({\mathbf T}) h = r({\mathbf V}) h - r({\mathbf T}) h
 $$
 and hence we see that $\cK \ominus \cH$ can be taken to have the form
 $$
   \cK \ominus \cH = \overline{\operatorname{span}} \{ r({\mathbf V}) h - r({\mathbf T}) h \colon h \in \cH, \, r \in
   \operatorname{Rat}(\Omega) \}.
 $$
 Then the computation, for $r, q \in \operatorname{Rat}(\Omega)$ and $h, h' \in \cH$,
 \begin{align*}
 & \langle q({\mathbf V}) (r({\mathbf V}) - r({\mathbf T})) h, h' \rangle_\cK
  = \langle q({\mathbf U}) ( r({\mathbf U}) - r({\mathbf T})) h, h' \rangle_{\widetilde \cK}  \\
 &  \quad \quad  = \langle q({\mathbf T}) (r({\mathbf T}) - r({\mathbf T}) ) h, h' \rangle_\cH = 0
\end{align*}
shows that $\cK \ominus \cH$ is invariant under $q(V)$ for any $q \in \operatorname{Rat}(\Omega)$, i.e.,
that $\cH$ is invariant under $q(V)^*$  for any $q \in \operatorname{Rat}(\Omega)$.  Once this is established
the next computation, for $h, h' \in \cH$ and $q \in \operatorname{Rat}(\Omega)$,
$$
 \langle q({\mathbf V})^* h, h' \rangle_\cH = \langle h, q({\mathbf V}) h' \rangle_\cK =
 \langle h, q({\mathbf U}) h' \rangle_{\widetilde \cK} = \langle h, q({\mathbf T}) h'  \rangle_\cH
 = \langle q({\mathbf T})^* h, h' \rangle_\cH
 $$
 shows that $q({\mathbf V})^*|_\cH = q({\mathbf T})^*$ for all $q \in \operatorname{Rat}(\Omega)$, and we conclude
 that ${\mathbf V}$ is an $\Omega$-isometric  lift of ${\mathbf T}$.
 Conversely, if ${\mathbf T}$ on $\cH$ has a $\Omega$-isometric lift ${\mathbf V}$ on $\cK \supset \cH$,
 by definition of $\Omega$-isometry it follows that ${\mathbf V}$ in turn has an $\Omega$-unitary extension
 ${\mathbf U}$ on $\widetilde \cK \supset \cK$.  It is now a simple check to see that ${\mathbf U}$ serves as an
 $\Omega$-unitary dilation of ${\mathbf T}$.

 Let us say that {\em rational dilation holds for $\Omega$} if any commutative operator tuple ${\mathbf T} =
 (T_1, \dots, T_d)$ having $\Omega$ as a spectral set has a $\Omega$-unitary dilation.  By the discussion above,
 an equivalent formulation is that rational dilation holds for $\Omega$ if any commutative operator tuple ${\mathbf T} =
 (T_1, \dots, T_d)$ having $\Omega$ as a spectral set has a $\Omega$-isometric lift.  In the present paper
 we shall work with the $\Omega$-isometric-lift formulation rather than the  $\Omega$-unitary-dilation
 formulation.

 The theorem of
 Sz.-Nagy says that rational dilation holds for the unit disk ${\mathbb D}$. For single-variable
 domains $\Omega \subset {\mathbb C}$, it is known that rational dilation holds if $\Omega$ is a singly or
 doubly connected domain \cite{agler-ann}, but can fail if $\Omega$ contains three or more holes
 \cite{AHR, DritMcCu}.
 For domains contained
 in higher-dimensional Euclidean space, the And\^o dilation theorem \cite{ando}
 says that rational dilation holds for the bidisk ${\mathbb D^2}$, but the well-known result of Varopoulos
 \cite{Varopoulos}
 tells us that rational dilation can fail for polydisks ${\mathbb D}^d$ of dimension $d \ge 3$.

 With original motivation coming from the problem of $\mu$-synthesis in robust control (see \cite{DP}),
 these issues have been addressed for other types of domains in ${\mathbb C}^2$ and ${\mathbb C}^3$:
 specifically, the symmetrized bidisk
 $$
 \Gamma = \{ (s, p) \in {\mathbb C}^2 \colon s = (\lambda_1 + \lambda_2), \, p = \lambda_1 \lambda_2 \text{ for some }
 (\lambda_1, \lambda_2) \in {\mathbb D}^2 \},
 $$
 and a domain in ${\mathbb C}^3$ called the {\em tetrablock} and denoted by ${\mathbb E}$:
 \begin{align*}
  {\mathbb E} =  & \{ (x_1, x_2, x_3) \in {\mathbb C^3} \colon \exists \, A = \begin{bmatrix} a_{11} & a_{12}
  \\ a_{21} & a_{22} \end{bmatrix} \in {\mathbb C}^{2 \times 2}   \\
  & \text{with } \| A \| < 1 \text{ and } x_1 = a_{11}, \,
  x_2 = a_{22}, \, x_3 = \det A \}.
 \end{align*}
 It was shown some time ago that rational dilation does hold for the symmetrized bidisk \cite{AY},  and
 current conventional
 wisdom is that rational dilation in general fails for the tetrablock in view of the work in \cite{Spal}.
 The strategy of \cite{Spal} was to identify
 some necessary conditions for a tetrablock contraction to have a tetrablock-isometric lift and then produce a
 a concrete tetrablock contraction which violates a particular one of these supposed necessary conditions.
 However we here present a tetrablock contraction which does have a tetrablock-unitary dilation and at the same time
 violates this supposed necessary condition,
 thus showing that this supposed necessary condition is not necessary after all.  We also present some
 alternative necessary conditions for existence of a tetrablock-isometric lift, but have not been able to
 produce an example of a tetrablock contraction which violates any of these alternative necessary conditions.
 Until additional progress is made, it appears that at present whether rational dilation holds for the tetrablock is
 an open question.

 \section{The rational dilation problem for tetrablock contractions: necessary conditions and sufficient conditions}

As a matter of notation, for any Hilbert-space contraction operator $X$, we let $D_X : = (I - X^* X )^{\frac{1}{2}}$
denote the {\em defect operator} of $X$, and we let $\cD_X = \overline{\operatorname{Ran}}\, D_X$
denote the closure of the range of $D_X$.

A major breakthrough on the structure  of tetrablock contractions was the discovery of the
basic invariant called the {\em fundamental operators} for the tetrablock contraction.

\begin{thm}  {\rm (See  \cite[Section 4]{Bhat-Tet}.)} \label{Thm:BhatTet}
Let $(T_1,T_2,T)$ be a tetrablock contraction on a Hilbert space $\cH$. Then there exist unique operators
$F_1,F_2$ in $\cB(\cD_T)$ such that for all $z\in\overline{\mathbb{D}}$, the numerical radius of
$F_1+zF_2$ is at most one and
\begin{align}\label{BhatFund}
T_1-T_2^*T=D_TF_1D_T \text{ and }T_2-T_1^*T=D_TF_2D_T.
\end{align}
Moreover, $F_1,F_2$ are the only bounded linear operators in $\cB(\cD_P)$ that satisfy
\begin{align}\label{FundRel}
\begin{cases}
D_T T_1=F_1D_T+F_1^*D_TT\\
D_T T_2=F_2D_T+F_1^*D_TT.
\end{cases}
\end{align}
and are referred to as the {\em fundamental operators} for the tetrablock contraction  $(T_1,T_2,T)$.
\end{thm}

The following result gives several equivalent more convenient characterizations of tetrablock isometries.

\begin{thm}  (See \cite[Theorem 5.7]{Bhat-Tet}.)  \label{TetIsoCharc}
Let $(V_1,V_2,V)$ be a triple of commuting contractions on a Hilbert space. Then the following are equivalent:
\begin{enumerate}
  \item $(V_1,V_2,V)$ is a tetrablock-isometry;
  \item $(V_1,V_2,V)$ is a tetrablock contraction with $V$ an isometry;
  \item $V_1=V_2^*V$, $V_2$ is a contraction, and $V$ is an isometry;
  \item $V_1=V_2^*V$, the spectral radii of $V_1$ and $V_2$ are at most one and $V$ is an isometry.
\end{enumerate}
\end{thm}

\subsection{Necessary conditions for existence of a tetrablock-isometric lift}

\begin{proposition}
Let $(T_1,T_2,T)$ be a tetrablock contraction and $F_1,F_2$ in $\cB(\cD_T)$ be its fundamental operators. Each
of the following conditions is necessary for $(T_1,T_2,T)$ to have a tetrablock-isometric lift:
\begin{enumerate}
  \item The pair $(F_1,F_2)$ has a joint Halmos dilation to a commuting subnormal pair $(S_1,S_2)$,
  i.e., there exists an isometric embedding $\Lambda$ of  $\cD_T$ into a larger space $\cF$ so that
  $F_i = \Lambda^* S_i \Lambda$ for $i  = 1,2$ where  $(S_1, S_2)$ can be extended to a commuting
  normal pair $(N_1, N_2)$ with  joint spectrum $\sigma(N_1,N_2)$ contained in the unuion of 2-tori
  $\{(z_1,z_2):|z_1|=|z_2|\leq 1\}$.

  \item $(F_1^* D_T T_1 - F_2^* D_T T_2)|_{\operatorname{Ker}D_T}=0$;

  \item $(F_1^* F_2^* - F_2^* F_1^*)D_T T|_{\operatorname{Ker}D_T}=0$.
\end{enumerate}
\end{proposition}

\begin{proof}
Let $(V_1,V_2,V)$ on $\cK$ be a tetrablock-isometric lift of $(T_1,T_2,T)$.  It is known that the tetrablock is
polynomially convex \cite{awy}. Hence it suffices to work with polynomials rather than the full algebra
$\operatorname{Rat}({\mathbb E})$ and without loss of generality we can assume that
\begin{align}\label{MinApace}
\cK=\overline{\operatorname{span}}\{V_1^{m_1}V_2^{m_2}V^mh:h\in\cH\,\text{ and }m_1,m_2,m\geq0\},
\end{align} and hence that
\begin{align}\label{restriction}
(V_1^*,V_2^*,V^*)|_{\cH}=(T_1^*,T_2^*,T^*).
\end{align}
With respect to the decomposition $\cK=\cH\oplus (\cK\ominus\cH)$ let $V_1,V_2$ and $V$ have the $2\times 2$ block operator matrix given by
\begin{align}\label{Blocks}
V_j=\begin{bmatrix}
      T_j & 0 \\
      C_j & S_j
    \end{bmatrix} \text{ for }j=1,2 \text{ and }
    V=\begin{bmatrix}
      T & 0 \\
      C & S
    \end{bmatrix}.
\end{align}
We read off from $V$ being an isometry that the entries $T, S$ in its  $2\times 2$ block-operator matrix decomposition
 satisfy  $T^*T+C^*C=I_{\cH}$.  This in turn readily implies that there exists an isometry $\Lambda:\cD_T\to\cH^\perp$ such that
\begin{align}\label{L}
\Lambda D_T=C.
\end{align}
Since $(V_1,V_2,V)$ on $\cK$ is a tetrablock isometry, part (3) of Theorem \ref{TetIsoCharc} tells us
 $V_1=V_2^*V$. By the block operator-matrix representations (\ref{Blocks}), this in turn translates to
\begin{align*}
\begin{bmatrix}
      T_1 & 0 \\
      C_1 & S_1
    \end{bmatrix} =\begin{bmatrix}
      T_2^* & C_2^* \\
      0 & S_2^*
    \end{bmatrix} \begin{bmatrix}
      T & 0 \\
      C & S
    \end{bmatrix}
    =\begin{bmatrix}
      T_2^*T+C_2^*C & C_2^*S \\
      S_2^*C & S_2^*S
    \end{bmatrix}
\end{align*}
which further implies that
\begin{align}\label{Eqn1}
  T_1-T_2^*T=C_2^*C, \quad C_2^*S=0\quad  \text{and}\quad C_1=S_2^*C.
\end{align}
Multiply the equation $V_1=V_2^*V$ from part (3) of  Theorem \ref{TetIsoCharc} on the left by $V^*$, use that
$V$ commutes with $V_1$ and that $V$ is an isometry, and then take adjoints to get
$$
V_2=V_1^*V.
$$
A similar computation gives the identity formally obtained from \eqref{Eqn1} by interchanging indices:
\begin{align}\label{Eqn2}
  T_2-T_1^*T=C_1^*C, \quad C_1^*S=0\quad  \text{and}\quad C_2=S_1^*C.
\end{align}

Now we have all the information needed to prove part (1). By the first equation in (\ref{Eqn1}) and the last equation
 in (\ref{Eqn2}) we have
\begin{align}   \label{proof1}
D_TF_1D_T=T_1-T_2^*T=C_2^*C=C^*S_1C=D_T\Lambda^*S_1\Lambda D_T \text{ (by \eqref{L}).}
\end{align}
By the uniqueness of the fundamental operators we have $F_1=\Lambda^*S_1\Lambda $. A similar computation shows
that $F_2=\Lambda^*S_2\Lambda $. Note that the triple $(S_1,S_2,S)$ is a tetrablock isometry, since
$(S_1,S_2,S)=(V_1,V_2,V)|_{\cH^\perp}$ and $(V_1,V_2,V)$ is a tetrablock isometry.
Hence $(S_1, S_2, S)$ has a tetrablock unitary extension, say
$(N_1, N_2, N)$.  By definition of a tetrablock unitary, the Taylor joint spectrum of $(N_1, N_2, N)$ is contained
in the distinguished boundary of the tetrablock.  By Theorem 7.1 in \cite{AHR}, this distinguished boundary  is
consists of the seet  $\{ (x_1, x_2, x_3) \in {\mathbb C}^3  \colon  x_1 = \overline{x_2} x_3, \, |x_2| \le 1, \, |x_3| = 1\}$.
By now ignoring the third component, we see that the joint spectrum of the commuting normal pair $(N_1, N_2)$ is contained in the union of $2$-tori $\{ (z_1, z_2) \colon |z_1| = |z_2| \le 1\}$, and item (1) follows.

Since $V_1$ and $V_2$ commute, we may equate the $(2,1)$-entry of $V_1 V_2$ with the
$(2,1)$-entry of $V_2 V_1$ to arrive at
\begin{equation} \label{Eqn3}
 C_1T_2+S_1C_2=C_2T_1+S_2C_1.
\end{equation}
After rearranging terms and using the last equations in (\ref{Eqn1}) and (\ref{Eqn2}),  we get
\begin{align*}
  S_1^*CT_1-S_2^*CT_2=(S_1S_1^*-S_2S_2^*)C.
\end{align*}
After multiplying by $\Lambda^*$ on the left and using (\ref{L}), we get
\begin{align}\label{proof2}
F_1^*D_TT_1-F_2^*D_TT_2=\Lambda^*(S_1S_1^*-S_2S_2^*)\Lambda D_T,
\end{align}
This completes the proof of  item (2).

Finally, we invoke equations (\ref{FundRel}) to get that
\begin{align}\label{proof3}
\nonumber \Lambda^*(S_1S_1^*-S_2S_2^*)\Lambda D_T&=F_1^*D_TT_1-F_2^*D_TT_2\\
\nonumber&=F_1^*(F_1D_T+F_2^*D_TT)-F_2^*(F_2D_T+F_1^*D_TT)\\
&=(F_1^*F_1-F_2^*F_2)D_T+(F_1^*F_2^*-F_2^*F_1^*)D_TT.
\end{align}This not only establishes item (3) but also shows that item (2) and item (3) are equivalent.
\end{proof}

\subsection{Sufficient conditions for the existence of an tetrablock-isometric lift}

The following set of sufficient conditions for rational dilation of a tetrablock contraction was obtained
by Bhattacharyya.

\begin{thm}  \label{T:Bhat-suf} {\rm (See \cite[Theorem 6.1]{Bhat-Tet}.)}
Let $(T_1, T_2, T)$ be a tetrablock contraction with fundamental operators $F_1$ and $F_2$.  Then a
sufficient condition for $(T_1, T_2, T)$ to have a tetrablock-isometric lift is that
\begin{enumerate}
\item $F_1 F_2 = F_2 F_1$, and
\item $F_1^* F_1 - F_1 F_1^* = F_2^* F_2 - F_2 F_2^*$.
\end{enumerate}
\end{thm}

There is a class of tetrablock contractions which always dilate, namely, those of the form $(T_1, T_2, T_1 T_2)$
with $(T_1, T_2)$ a pair of commuting contractions.  We begin with
 the following result which also appears in \cite{sauAndo} (Lemma 32).

\begin{thm}\label{T:thekeylem}
Let $T_1$ and $T_2$ be two commuting contractions on a Hilbert space $\mathcal H$ and $T=T_1T_2$. Then the triple $(T_1,T_2,T)$ is a tetrablock contraction on $\mathcal H$.
\end{thm}

\begin{proof}
The proof is a simple application of Ando's Theorem \cite{ando}. Define the map $\pi:\mathbb D\times \mathbb D\to \mathbb C^3$ by
$$\pi(z_1,z_2)=(z_1,z_2,z_1z_2).$$
Then by the definition of the tetrablock, it follows that Ran$(\pi) \subset  \mathbb E$. Now let $f$ be any polynomial in three variables. Now by Ando's theorem,
$$
\|f\circ \pi (T_1,T_2)\|\leq \|f\circ\pi\|_{\infty,  \mathbb D^2}\leq \|f\|_{\infty,  \mathbb E},
$$ which proves the lemma.
\end{proof}

It is now an application of the And\^o dilation theorem to see that any such tetrablock contraction
has a tetrablock-isometric lift as follows.

\begin{thm}  \label{T:pair-dilate}
For any pair $(T_1,T_2)$ of commuting contractions on a Hilbert space, the tetrablock contraction
$(T_1,T_2,T_1T_2)$ always has a tetrablock-isometric lift.
\end{thm}

\begin{proof}
Let $(V_1,V_2)$ be an And\^o isometric lift for $(T_1,T_2)$. Then $(V_1,V_2, V_1V_2)$ is a triple of
commuting isometries which is a lifting of $(T_1,T_2,T_1T_2)$.  By condition (3) in
Theorem \ref{TetIsoCharc}, one sees that the triple $(V_1,V_2,V_1V_2)$ is a tetrablock isometry.
\end{proof}

\section{Tetrablock contractions with special structure}

To get more tractable examples to work with, in this section we consider tetrablock contractions $(T_1, T_2, T)$
where $T$ is a partial isometry.  For clarity of the results which follow, we assume only hypotheses needed to the
validity of the particular result at hand.

The following structure of partial isometries is well known. We omit the striaghtforward proof.
\begin{lemma}\label{Lem:PI}
Let $T$ be a contraction on a Hilbert space $\cH$. Then $T$ is a partial isometry if and only if there exists a decomposition
$\cH=\cH_1\oplus\cH_2$ such that
\begin{align}\label{PI}
T=\begin{bmatrix}
    Z & 0
  \end{bmatrix}:\cH_1\oplus\cH_2\to\cH
\end{align}for some isometry $Z:\cH_1\to\cH$.
\end{lemma}

In the next result we consider an operator-triple $(T_1, T_2, T)$  which satisfies only some of the properties of a
tetrablock contraction.

\begin{proposition}\label{P:varPal}
Let $T$ be a partial isometry and $(T_1,T_2)$ be a pair of contractions on $\cH$. Suppose there exist two operators $F_1,F_2$ in $\cB(\cD_T)$ such that
\begin{align}\label{FundEqns}
T_1-T_2^*T=D_TF_1D_T \text{ and }T_2-T_1^*T=D_TF_2D_T.
\end{align}
Then
\begin{enumerate}
\item[(1)]$\operatorname{Ker} T$ is jointly invariant under $(T_1,T_2)$ and,
\item[(2)] if we denote the restriction $(T_1,T_2)|_{\operatorname{Ker T}}$ by
$(D_1,D_2)$, then
\begin{enumerate}
\item[(a)] $F_1F_2=F_2F_1 \text{ if and only if }D_1D_2=D_2D_1$\text{ and}
\item[(b)] $F_1^*F_1-F_1F_1^*=F_2^*F_2-F_2F_2^* \text{ if and only if }D_1^*D_1-D_1D_1^*=D_2^*D_2-D_2D_2^*.$
\end{enumerate}
\end{enumerate}
\end{proposition}

\begin{proof}
We first observe that since $T$ is a partial isometry, by Lemma \ref{Lem:PI},
\begin{align}\label{DefectT}
D_T=\begin{bmatrix} I_{\operatorname{Ran} T^*} & 0 \\
    0 & I_{\operatorname{Ker}T} \end{bmatrix}
    -  \begin{bmatrix} I_{\operatorname{Ran} T^*} & 0 \\
    0 & 0 \end{bmatrix}
  =
  \begin{bmatrix}
    0 & 0\\  0 & I_{\operatorname{Ker}T} \end{bmatrix},
\end{align}
which implies that $\cD_T=0\oplus\operatorname{Ker}T$.
Therefore $F_1$, $F_2$, being operators on $\cD_T$, are of the form
$$
F_1=\begin{bmatrix}
    0 & 0\\
    0 & L_1
  \end{bmatrix}\text{ and }
  F_2=\begin{bmatrix}
    0 & 0\\
    0 & L_2
  \end{bmatrix}
$$
for some operators $L_1,L_2$ on $\operatorname{Ker}T$. Let $T_1$, $T_2$ be of the following form
with respect to the decomposition $\cH=\operatorname{Ran} T^* \oplus\operatorname{Ker} T$:
\begin{align}\label{FormT1&T2}
T_j=\begin{bmatrix}
    A_j & B_j\\
    C_j & D_j
  \end{bmatrix} \text{ for }j=1,2
\end{align}and
\begin{align}\label{FormT}
T=\begin{bmatrix} Y & 0\\ X & 0 \end{bmatrix} \colon \operatorname{Ran}  T^* \oplus \operatorname{Ker} T
\to   \operatorname{Ran}T^* \oplus \operatorname{Ker} T.
\end{align}
Because $T$ is a partial isometry, by Lemma \ref{Lem:PI}, $Z=\begin{bmatrix}Y & X\end{bmatrix}^T$ is an isometry. Now
\begin{align*}
&T_1-T_2^*T=D_TF_1D_T\\
\Leftrightarrow &\begin{bmatrix}
    A_1 & B_1\\
    C_1 & D_1
  \end{bmatrix} -\begin{bmatrix}
    A_2^* & C_2^*\\
    B_2^* & D_2^*
  \end{bmatrix} \begin{bmatrix}
    Y & 0\\
    X & 0
  \end{bmatrix}= \begin{bmatrix}
    0 & 0\\
    0 & I_{\operatorname{Ker}T}
  \end{bmatrix}
  \begin{bmatrix}
    0 & 0\\
    0 & L_1
  \end{bmatrix}
   \begin{bmatrix}
    0 & 0\\
    0 & I_{\operatorname{Ker}T}
  \end{bmatrix} \\
\Leftrightarrow&
\begin{bmatrix}
    A_1-(A_2^*Y+C_2^*X) & B_1\\
    C_1-(B_2^*Y+D_2^*X) & D_1
  \end{bmatrix}=\begin{bmatrix}
    0 & 0\\
    0 & L_1
  \end{bmatrix}.
\end{align*}Therefore by the above computation and a similar treatment with $T_2-T_1^*T=D_TF_2D_T$ we obtain the following for $(i,j)=(1,2)$ or $(2,1)$:
\begin{align}\label{Eqns}
\begin{cases}
(i)\; L_i=D_i, \quad \quad \;\;(iii)\; A_i=A_j^*Y+C_j^*X=
            \begin{bmatrix}
              A_j^* & C_j^*
            \end{bmatrix}\begin{bmatrix}
              Y & X
            \end{bmatrix}^T,\\
(ii)\, B_i=0, \quad \quad \quad (iv)\; C_i=D_j^*X.
\end{cases}
\end{align}Therefore for $(i,j)=(1,2)$ or $(2,1)$,
\begin{align}\label{T1&T2}
T_i=\begin{bmatrix}
    Z_j^*Z & 0\\
    D_j^*X & D_i
  \end{bmatrix},
\end{align}where for $(i,j)=(1,2)$ or $(2,1)$, $Z_i=\begin{bmatrix}Z_j^*Z & D_j^*X\end{bmatrix}^T$ is the first column of $A_i$. Equation (\ref{Eqns}) (i) implies that
\begin{align}\label{FundOps}
F_1=\begin{bmatrix}
    0 & 0\\
    0 & D_1
  \end{bmatrix}\text{ and }
  F_2=\begin{bmatrix}
    0 & 0\\
    0 & D_2
  \end{bmatrix},
\end{align}
and (1) and (2) now follow.
\end{proof}

Note that if the pair $(T_1,T_2)$ of contractions in Proposition \ref{P:varPal} is commuting, then
 $T_1T_2|_{\operatorname{Ker}T}=T_2T_1|_{\operatorname{Ker}T}$. Hence part (1) in Proposition \ref{P:varPal} is automatic. Moreover, if the commuting pair $(T_1,T_2)$ of contractions is such that $(T_1,T_2,T)$ is a tetrablock contraction, then by Theorem \ref{Thm:BhatTet}, there exist two operators $F_1,F_2$ in $\cB(\cD_T)$ such that (\ref{FundEqns}) holds. Hence the following is a simple consequence of Proposition \ref{P:varPal}.

\begin{corollary}\label{CP:varPal}
Let $(T_1,T_2,T)$ be a tetrablock contraction on a Hilbert space $\cH$ with fundamental operators $F_1,F_2$. If $T$ is a partial isometry, then:
\begin{enumerate}
  \item $\operatorname{Ker}T$ is jointly invariant under $(T_1,T_2)$.
  \item $F_1F_2=F_2F_1$.
  \item $F_1^*F_1-F_1F_1^*=F_2^*F_2-F_2F_2^*$  if and only if
  $D_1^*D_1-D_1D_1^*=D_2^*D_2-D_2D_2^*$, where $(D_1,D_2)=(T_1,T_2)|_{\operatorname{Ker}T}$.
\end{enumerate}
\end{corollary}

We next present an example of a tetrablock contraction which has a tetrablock-isometric lift for which
condition (3) (in either of the equivalent forms) in Corollary \ref{CP:varPal} fails.  Thus the set of sufficient
conditions for existence of a tetrablock-isometric lift in Theorem \ref{T:Bhat-suf} fails in general to be also
necessary, even when the tetrablock contraction $(T_1, T_2, T)$ has the special form where $T$ is a partial
isometry.

\begin{example}  \label{E:counterexample} {\rm
Consider the following pair of contractions on $H^2\oplus H^2$:
\begin{align}\label{TheForm}
(T_1,T_2)=\left(\begin{bmatrix} 0 & 0 \\ I & 0 \end{bmatrix},\begin{bmatrix}T_z & 0 \\ 0& T_z \end{bmatrix}\right).
\end{align}
Note that
\begin{align}\label{commuting}
&T_1T_2=\begin{bmatrix} 0 & 0 \\ T_z & 0 \end{bmatrix}=T_2T_1=:T.
\end{align}
Therefore by Theorem \ref{T:thekeylem} the triple $(T_1,T_2,T)$ is a tetrablock contraction which moreover
has a tetrablock-isometric lift as a consequence of Theorem \ref{T:pair-dilate}.  From the form of $T$ in \eqref{commuting} we see that furthermore $T$ is a partial isometry since $T_z$ is an isometry, so Corollary
\ref{CP:varPal} applies to this choice of operator triple $(T_1, T_2, T)$.  As a first step toward computing
the fundamental operators $F_1$, $F_2$ for $(T_1, T_2, T)$, we compute the defect operator for $T$:
\begin{align}\label{defect}
 D_T^2=\begin{bmatrix} I_{H^2} & 0 \\ 0 & I_{H^2} \end{bmatrix}-\begin{bmatrix} 0 & T_z^* \\ 0 & 0 \end{bmatrix}\begin{bmatrix} 0 & 0 \\ T_z & 0 \end{bmatrix}=\begin{bmatrix} 0 & 0 \\ 0 & I_{H^2} \end{bmatrix}=D_T.
\end{align}
 Let us set
\begin{align}\label{TheFund}
(A_1,A_2):=(0,T_z).
\end{align}
The computations
\begin{align*}
T_1-T_2^*T=\begin{bmatrix} 0 & 0 \\ I_{H^2} & 0 \end{bmatrix}-\begin{bmatrix} T_z^* & 0 \\ 0 & T_z^* \end{bmatrix}\begin{bmatrix} 0 & 0 \\ T_z & 0 \end{bmatrix}&=\begin{bmatrix} 0 & 0 \\ 0 & 0 \end{bmatrix}\\
&=\begin{bmatrix} 0 & 0 \\ 0 & I_{H^2} \end{bmatrix}\begin{bmatrix} 0 & 0 \\ 0 & 0 \end{bmatrix}
\begin{bmatrix} 0 & 0 \\ 0 & I_{H^2} \end{bmatrix}=D_TA_1D_T,
\end{align*}
and
\begin{align}
T_2-T_1^*T=\begin{bmatrix} T_z & 0 \\ 0 & T_z \end{bmatrix}-\begin{bmatrix} 0 & I_{H^2} \\ 0 & 0 \end{bmatrix}\begin{bmatrix} 0 & 0 \\ T_z & 0 \end{bmatrix}&=\begin{bmatrix} 0 & 0 \\ 0 & T_z \end{bmatrix}\\
&=\begin{bmatrix} 0 & 0 \\ 0 & I_{H^2} \end{bmatrix}\begin{bmatrix} 0 & 0 \\ 0 & T_z \end{bmatrix}\begin{bmatrix} 0 & 0 \\ 0 & I_{H^2} \end{bmatrix}=D_TA_2D_T
\end{align}
then show that $(A_1, A_2)$ are equal to the fundamental operators $(F_1, F_2)$ for $(T_1, T_2, T)$.
Observe that
$$
 A_1^* A_1 - A_1 A_1^* = 0 \text{ while } A_2^* A_2 - A_2 A_2^* = I_{H^2} - T_z T_z^* = P_0 \ne 0
 $$
 (where $P_0$ is the orthogonal  projection of $H^2$ onto the constant functions in $H^2$), so the first form
 of condition (3) in Corollary \ref{CP:varPal} is violated.
}\end{example}

\begin{remark} \label{R:Pal}
{\em  In \cite{Spal} Pal considered the class of tetrablock contractions $(T_1, T_2, T)$ on a Hilbert space
decomposed as $\cH = \cH_1 \oplus \cH_1$ for another Hilbert space $\cH_1$ such that
\begin{enumerate}
\item[(i)] $\operatorname{Ker} D_T = \cH_1 \oplus \{0\}$  and  $\cD_T = \{0 \} \oplus \cH_1$,
\item[(ii)] $T (\cD_T) = \{0\}$  and $T \operatorname{Ker} D_T \subset \cD_T$.
\end{enumerate}
If $T$ is a Hilbert space operator satisfying the first of conditions (ii), then in particular
$$
  0 = T D_T^2 = T (I  - T^* T) = T - T T^* T,
$$
i.e., $T = T T^* T$ implying that $T^* T$ is a projection which is one of the equivalent conditions for $T$ to be
a partial isometry.
Thus the class of tetrablock contractions considered
in Corollary \ref{CP:varPal} includes the class of tetrablock contractions $(T_1, T_2, T)$ with $T$ satisfying
Pal's conditions (i), (ii) above.  In Proposition 4.5 of \cite{Spal}, Pal asserts that the two sufficient conditions
from Theorem \ref{T:Bhat-suf} on the fundamental operators $(F_1, F_2)$ for $(T_1, T_2, T)$ for
existence of a tetrablock-isometric lift, namely
\begin{enumerate}
\item[(P1)] $F_1 F_2 = F_2 F_1$, \\
\item[(P2)] $F_1^* F_1 - F_1 F_1^* = F_2^* F_2 - F_2 F_2^*$,
\end{enumerate}
are also necessary for the case where the tetrablock contraction$(T_1, T_2, T)$ satisfies conditions (i) and (ii)
above.  From Corollary \ref{CP:varPal} we see that
condition (P1) in fact holds for any tetrablock contraction of this special form, independently of whether or not it has a
tetrablock-isometric lift.  Furthermore, one can check that the tetrablock contraction given in
Example \ref{E:counterexample} can be represented on a Hilbert space of the form $\cH = \cH_1 \oplus \cH_1$
with Pal's conditions  (i) and (ii) satisfied.  As we have seen, this example is a tetrablock contraction having
a tetrablock-isometric lift which violates condition (P2).

We note that Example \ref{E:counterexample} is just one sample
of a general class of such counterexamples.  We let $(T_1, T_2)$ be a commuting pair of contraction operators
such that $(i) T = T_1 T_2$ is a partial isometry and (ii) the commuting pair
$(D_1, D_2) = (T_1, T_2)|_{\operatorname{Ker} T}$ is such that $D_1^* D_1 - D_1 D_2^* \ne
D_2^* D_2 D_2 - D_2 D_2^*$.  Then the tetrablock contraction $(T_1, T_2, T)$ does have a tetrablock-isometric
lift but fails to satisfy condition (P2).  The point of Example \ref{E:counterexample} is to show that this
class is nonempty.

In Section 5 of \cite{Spal}, Pal considers the following candidate for a tetrablock contraction not having a
tetrablock-isometric lift. Let $\cH=\cH_1\oplus\cH_1$, where $\cH_1=H^2(\mathbb{C}^2)\oplus H^2(\mathbb{C}^2)$,
\begin{align}\label{PalEg}
(T_1,T_2,T)=\left(\begin{bmatrix}
    0 & 0\\
    0 & J
  \end{bmatrix},\begin{bmatrix}
    0 & 0\\
    0 & 0
  \end{bmatrix},\begin{bmatrix}
    0 & 0\\
    Y & 0
  \end{bmatrix}\right),
\end{align}
where
\begin{align} \label{Entries}
&J:=  \begin{bmatrix} H & 0 \\ 0 & 0 \end{bmatrix} \text{ with }
H = (I - M_z M_z^*) \otimes H_1 \text{ and }
H_1 = \begin{bmatrix} 0 & \frac{1}{4} \\ 0 & 0 \end{bmatrix} \text{ on } {\mathbb C}^2, \\
&Y:    = \begin{bmatrix} 0 & M_z\\ I & 0 \end{bmatrix}  \text{ on } \cH_1.
\end{align}
We note that $Y$ is an isometry and that $D_T = D_T^2  = \left[ \begin{smallmatrix} 0 & 0 \\ 0 & I_{\cH_1}
\end{smallmatrix} \right]$.   The fact that $JY = 0$ has the consequence that  $(T_1, T_2, T)$ as in \eqref{PalEg}
is a commutative operator triple.  Once it is further checked that $(T_1, T_2, T)$ is a tetrablock contraction,
one can calculate the fundamental operators:
$$
F_1 = J, \quad F_2 = 0
$$
where here we identify $ \{0 \} \oplus \cH_1$ with $\cH_1$.
Pal then concludes that this tetrablock contraction does not have a tetrablock-isometric lift
by arguing that the fundamental
operators $F_1,F_2$ of $(T_1,T_2,T)$ do not satisfy condition (P2) above.

However, as noted above, Example \ref{E:counterexample} above shows that condition (P2) is not necessary
for the existence of a tetrablock-isometric lift after all, so additional evidence is required to verify that
this example fails to have a tetrablock-isometric lift.

 Indeed, the fact that $(F_1, F_2)$ fails to satisfy condition (P2)
can be seen as a consequence of the second formulation of condition (P2) in part (3) of Corollary \ref{CP:varPal}  as follows.   Since $Y$ is an isometry, it follows that
$T$ is a partial isometry. Hence Corollary \ref{CP:varPal} applies to $(T_1,T_2,T)$. Let us note that
$$
\operatorname{Ker} T = \operatorname{Ker} \begin{bmatrix} 0 & 0 \\ Y & 0 \end{bmatrix} =
\begin{bmatrix} \{0\} \\ H^2({\mathbb C}^2)) \end{bmatrix}.
$$
Hence
$$
D_1 = T_1|_{\operatorname{Ker} T}  = J, \quad D_2 = T_2|_{\operatorname{Ker} T}  = 0.
$$
From part (3) of Proposition \ref{P:varPal} we see directly that $(F_1,F_2$) does not satisfy condition (P2)
since $J^*J-JJ^*\neq 0$.
}\end{remark}

\noindent{\bf Acknowledgement.} This work was done when the second named author was visiting Virginia Tech as an SERB Indo-U.S. Postdoctoral Research Fellow. He wishes to thank the Department of Mathematics, Virginia Tech for all the facilities provided to him.

\end{document}